\documentclass[11pt]{amsart}

\usepackage{tikz}
\usetikzlibrary{matrix}
\usetikzlibrary{patterns, shapes}
\usetikzlibrary{arrows}
\usetikzlibrary{calc,3d}
\usetikzlibrary{decorations,decorations.pathmorphing, decorations.pathreplacing}
\usetikzlibrary{through}
\tikzset{ext/.style={circle, draw,inner sep=1pt},int/.style={circle,draw,fill,inner sep=1pt},nil/.style={inner sep=1pt}}
\tikzset{exte/.style={circle, draw,inner sep=3pt},inte/.style={circle,draw,fill,inner sep=3pt}}
\tikzset{diagram/.style={matrix of math nodes, row sep=3em, column sep=2.5em, text height=1.5ex, text depth=0.25ex}}
\tikzset{diagram2/.style={matrix of math nodes, row sep=0.5em, column sep=0.5em, text height=1.5ex, text depth=0.25ex}}

\usepackage{tikz-cd}

\usepackage{amssymb, amsmath, amscd, amsthm, graphicx, psfrag,rotating}
\usepackage{enumerate}
\usepackage[matrix,arrow,curve]{xy}
\xyoption{dvips}              
\usepackage{wasysym}
\usepackage{epstopdf}

\usepackage{amssymb, amsmath, amscd, amsthm,
color, epsfig
}
\usepackage[english]{babel}

\title[Mapping spaces of Swiss cheese operads]{Mapping spaces
of Swiss cheese operads}
\author{Victor Turchin}
\address{Kansas State University}
\email{turchin@ksu.edu}

\thanks{The author was partially supported by the Simons Foundation award  \#933026.}

\newtheorem{theorem}{Theorem}[section]

\newtheorem{lemma}[theorem]{Lemma}
\newtheorem{proposition}[theorem]{Proposition}

\newtheorem{mainthm}{Theorem}

\theoremstyle{definition}
\newtheorem{definition}[theorem]{Definition}
\newtheorem{remark}[theorem]{Remark}
\newtheorem{example}[theorem]{Example}

\newcommand{\Q}{{\mathbb{Q}}}

\newcommand{\R}{{\mathbb R}}

\newcommand{\SC}{{\mathcal{SC}}}
\newcommand{\E}{{\mathcal E}}
\newcommand{\F}{{\mathcal F}}
\newcommand{\FS}{{\mathcal{FS}}}

\newcommand{\TT}{{\mathrm{T}}}
\newcommand{\MM}{{\mathrm{M}}}

\newcommand{\Ebarmn}{{\overline{\mathrm{Emb}}}_\partial(D^m,D^n)}
\newcommand{\Embmn}{\mathrm{Emb}_\partial(D^m,D^n)}

\newcommand{\Ebar}{\overline{\mathrm{Emb}}}
\newcommand{\Emb}{\mathrm{Emb}}

\newcommand{\Map}{\mathrm{Map}}

\newcommand{\Dbar}{\overline{\mathrm{Diff}}}

\newcommand{\Op}{\mathrm{Op}}
\newcommand{\Aut}{\mathrm{Aut}}

\newcommand{\VV}{\mathrm{V}}

\newcommand{\hofiber}{\mathrm{hofiber}}

\newcommand{\calP}{\mathcal{P}}
\newcommand{\calF}{\mathcal{F}}
\newcommand{\calQ}{\mathcal{Q}}

\numberwithin{equation}{section}

\begin{document}

\sloppy

\begin{abstract}
We show that the color restriction map $\Op^h(\SC_m,\SC_n)\to \Op^h(\E_{m-1},\E_{n-1})$
from the derived mapping space of Swiss cheese operads to that of little discs operads, is
a weak homotopy equivalence. We explain how this can help in the study of disc concordance embedding spaces.

\end{abstract}

\maketitle

\section{Introduction}\label{s:intro}
The manifold functor calculus was invented by T.~Goodwillie and M.~Weiss in order to study embedding spaces~\cite{GW,Weiss}. In its heart this approach replaces embedding spaces by induced maps of configuration spaces.
In its modern formulations the calculus approximations are described in terms of module maps over the little discs operad~\cite{AT1,BW_man_calc,KrKu2,Turchin}.
The approach can be generalized to spaces of {\it proper} embeddings (that send boundary to boundary), see~\cite[Section~9]{BW_man_calc}, for which instead of modules over the little discs operad one would need to use modules over the Swiss cheese operad. In other words,  one would need to look at the configuration spaces of points both in the interior and in the boundary of the source and target manifolds, whose interaction is encoded
by the Swiss Cheese operad. The main result of the paper is not a generalization of any previous  result, but rather a precursor and a key for  such generalizations, for example, for all findings from~\cite{AT1,BW_conf_cat,DT,FTW,Weiss1}
as it is outlined below.

The space $\Embmn$, $n\geq m$, of smooth disc embeddings relative to the boundary and their modulo immersion version
\[
\Ebarmn:=\hofiber\left(\Embmn\to\Omega^m\VV_{n,m}\right)
\]
are of particular importance as they encode how any general embedding can be deformed locally. These spaces were topic of active study in recent years. In the 
context of the Goodwillie-Weiss calculus~\cite{GW,Weiss}, it was shown that
\begin{gather}
\Ebarmn\simeq_{n-m\geq 3}\TT_\infty\Ebarmn\simeq\Omega^{m+1}\Op^h(\E_m,\E_n);\label{eq:deloop1}\\
\TT_k\Ebarmn\simeq\Omega^{m+1}\TT_k\Op^h(\E_m,\E_n),\,\,\,\, k\geq 1;\label{eq:deloop2}
\end{gather}
where $\TT_k\Ebarmn$ denotes the $k$th Goodwillie-Weiss approximation to $\Ebarmn$; $\TT_\infty\Ebarmn$
is the limit of the tower $\TT_\bullet\Ebarmn$; $\Op^h(-,-)$ and $\TT_k\Op^h(-,-)$ denote the derived mapping spaces
of operads and $k$-truncated operads, respectively; see~\cite{AT1,BW_conf_cat,DT,Weiss1}. One can ask if similar results
hold for the spaces $\Emb_{\partial_+}(D^m,D^n)$, $n\geq m$, of disc concordance embeddings i.e.,
smooth embeddings $f\colon D^m\hookrightarrow D^n$, such that $f^{-1}(\partial D^n)=\partial D^m$ and which
coincide with the standard equatorial inclusion on a fixed half of the boundary $D^{m-1}_+\subset S^{m-1}=\partial
D^m$. Note that these spaces encode how  any proper embedding $M\hookrightarrow N$ can be deformed near a boundary point. It is natural to expect that the modulo immersion version $\Ebar_{\partial_+}(D^m,D^n)$
of this space is related in the same way to the derived mapping space  $\Op^h(\SC_m,\SC_n)$ of Swiss
cheese operads. However, the relation is not so straightforward. One has natural color restriction maps
\begin{gather}
 \Op^h(\E_{m-1},\E_{n-1})\xleftarrow{\,\simeq\,} \Op^h(\SC_m,\SC_n)\xrightarrow{R_\infty}\Op^h(\E_m,\E_n);
\label{eq:zig1}\\
\TT_k \Op^h(\E_{m-1},\E_{n-1})\xleftarrow{\,\simeq\,} \TT_k\Op^h(\SC_m,\SC_n)\xrightarrow{R_k}\TT_k\Op^h(\E_m,\E_n), \,\,\,\, k\geq 1.
\label{eq:zig2}
\end{gather}
In this paper we show that the left maps in~\eqref{eq:zig1} and~\eqref{eq:zig2} are weak equivalences, while 
we conjecture that
\begin{equation}\label{eq:conj}
\TT_k\Ebar_{\partial_+}(D^m,D^n)\simeq \Omega^m\hofiber(R_k),\,\,\,\, 1\leq k\leq \infty.
\end{equation}
Note that $\hofiber(R_\infty)$ can also be thought of as the derived space $\Op^h(\SC_m,\SC_n\,\,\mathrm{mod}\,\,\E_m)$
of operad maps $\SC_m\to\SC_n$ that coincide with the standard map $\E_m\to\E_n$ when restricted on the second color.

Looking at the same problem slightly differently, one has a fiber sequence
\begin{equation}\label{eq:fib_seq}
\Ebarmn\to \Ebar_{ \partial_+}(D^m,D^n)\to \Ebar_\partial(D^{m-1},D^{n-1}).
\end{equation}
Its extension to the left is the so-called {\it graphing map}
\[
\Omega \Ebar_\partial(D^{m-1},D^{n-1})\to \Ebarmn,
\]
whose $(m+1)$st delooping  in case $n-m\geq 3$, is expected to be a map
\[
\Op^h(\E_{m-1},\E_{n-1}) \to \Op^h(\E_{m},\E_{n}),
\]
 induced by the Boardman-Vogt tensor product with $\E_1$. 
 We conjecture that the zigzag~\eqref{eq:zig1} models this map.

The zigzags \eqref{eq:zig1} and \eqref{eq:zig2} have their rational versions in which $\E_{n-1}$, $\E_n$,
$\SC_n$ are replaced by rationalization  $\E_{n-1}^\Q$, $\E_n^\Q$, $\SC_n^\Q$, see Remark~\ref{r:sc}.
We believe that the methods in~\cite{FTW} (where the rational homotopy type of $\Ebarmn$ and $\Embmn$, $n-m\geq 3$, is computed) together with explicit rational models of $\SC_m$ and  $\SC_n$ \cite{LW,Will} (see also~\cite{IdVas,Liv}) can similarly be used to understand the rational homotopy type of 
$\Ebar_{\partial_+}(D^m,D^n)$ and $\Emb_{\partial_+}(D^m,D^n)$.

The equivalences in~\eqref{eq:zig1}-\eqref{eq:zig2} and conjecture~\eqref{eq:conj} are of particular interest when $n=m$. In this case we get that the color restriction map $\Aut^h(\SC_n)\xrightarrow{\simeq}\Aut^h(\E_{n-1})$
is an equivalence, while we conjecture that the limit of the Goodwillie-Weiss tower for the group of disc pseudoisotopies  $\Dbar_{\partial_+}(D^n):=\Ebar_{\partial_+}(D^n,D^n)$ deloops as
  $$
\TT_\infty\Dbar_{\partial_+}(D_n)\simeq \Omega^{n}\Aut^h(\SC_n\,{\mathrm{mod}}\, \E_n),
$$
 where $\Aut^h(\SC_n\,{\mathrm{mod}}\, \E_n)$ is the group of homotopy automorphisms of $\SC_n$ preserving its second-color part $\E_n$. The relation between $\Dbar_\partial(D_n)$ and $\TT_\infty\Dbar_\partial(D^n)\simeq\Omega^{n+1}\Aut^h(\E_n)$ was studied by Horel, Krannich, Kupers, and Randal-Williams~\cite{HKK,KrKu,KraRand,KupRand,Randal}. The theory of diffeomorphism groups of odd-dimensional discs is substantially more difficult than that of even-dimensional discs~\cite{KraRand,KupRand,Randal}.  The group $\Dbar_{\partial_+}(D^n)$ of pseudoisotopies provides an important connection between the two thanks to the fiber sequence~\eqref{eq:fib_seq}~\cite{Randal}.

\subsection{Main theorem}\label{ss:main_thm}
In the paper we prove the following result.

\begin{mainthm}\label{thm}
For any $m\geq 1$ and any appropriate reduced two-colored symmetric topological operad~$\calP$, the color restriction maps
\begin{gather}
\TT_k\Op^h(\SC_m,\calP)\xrightarrow{\,\simeq\,} \TT_k\Op^h(\E_{m-1},\calP_1), \,\,\,\, k\geq 1;\label{eq:th1}\\
\Op^h(\SC_m,\calP)\xrightarrow{\,\simeq\,} \Op^h(\E_{m-1},\calP_1) \label{eq:th2}
\end{gather}
are weak equivalences.
\end{mainthm}

For a 2-colored symmetric operad $\calP$, we denote by $\calP(i,j)_r$ its spaces of operations
with $i$ inputs of the first color, $j$ inputs of the second color and the output of color $r$; $i,j\geq 0$, $r\in\{1,2\}$.
We denote by $\calP_1:=\{\calP_1(i):=\calP(i,0)_1,\, i\geq 0\}$ and $\calP_2:=\{\calP_2(j):=\calP(0,j)_2,\, j\geq 0\}$
its one-colored suboperads. 

A 2-colored operad $\calP$ is said {\it reduced} if both of its arity zero components are singletons: $\calP(0,0)_1=*=
\calP(0,0)_2$.

\begin{definition}\label{d:appropr}
A reduced 2-colored topological operad $\calP$ is  {\it appropriate} if
\renewcommand{\theenumi}{\alph{enumi}}
\begin{enumerate}
\item $\calP(0,1)_1\simeq *$;
\item the maps $\calP(i,j)_1\xrightarrow{\,\simeq\,}\calP(i,0)_1\times\calP(0,j)_1$ induced by compositions with the arity zero singleton components are equivalences for $i,j\geq 1$;
\item the composition maps $\circ_1\colon \calP(0,1)_1\times \calP(0,j)_2\xrightarrow{\,\simeq\,}\calP(0,j)_1$,
$(x,y)\mapsto x\circ_1 y$, are equivalences for $j\geq 2$.
\end{enumerate}

\end{definition}

Note that (a) and (c) imply that for any $x\in \calP(0,1)_1$, the composition with $x$ map
\begin{equation}\label{eq:gamma_x}
\calP(0,j)_2\xrightarrow{\,\simeq\,}\calP(0,j)_1,\,\,\, y\mapsto x\circ_1 y
\end{equation}
is a weak equivalence.

\begin{remark}\label{r:m1}
In case $m=1$, the statement of Theorem~\ref{thm} holds  for any 2-colored operad $\mathcal P$ satisfying a slightly weaker condition. Namely,  condition~(b)
in Definition~\ref{d:appropr} needs to hold only for $i=1$ and any $j\geq 1$. Note that equivalences \eqref{eq:th1} and \eqref{eq:th2} for $m=1$ imply $\TT_k\Op^h(\SC_1,{\mathcal P})\simeq *\simeq \Op^h(\SC_1,{\mathcal P})$.
\end{remark}

\begin{remark}\label{r:sc}
It is easy to see that both the Swiss cheese operad $\SC_n$ and its rationalization $\SC_n^\Q$ are reduced 
appropriate operads. This implies that the left maps in \eqref{eq:zig1} and \eqref{eq:zig2} and the
maps 
\begin{gather*}
\TT_k\Op^h(\SC_m,\SC_n^\Q)\xrightarrow{\,\simeq\,} \TT_k\Op^h(\E_{m-1},\E_{n-1}^\Q), \,\,\,\, k\geq 1;\\
\Op^h(\SC_m,\SC_n^\Q)\xrightarrow{\,\simeq\,} \Op^h(\E_{m-1},\E_{n-1}^\Q). 
\end{gather*}
are weak equivalences.
\end{remark}

\begin{remark}\label{r:other}
We are not aware of any other {\it appropriate} operads besides the Swiss cheese ones (and their localizations). However, there are many geometric operads satisfying condition~(b) of Definition~\ref{d:appropr} that appear if one studies embeddings respecting stratification. Even though Theorem~\ref{thm} is not applicable to such operads, some techniques of its proof are. For example, one can consider  spaces of embeddings of manifolds $f\colon M\hookrightarrow N$, such that $f^{-1}(L)=\partial M$, where $L$ is a fixed proper submanifold in~$M$. The corresponding operads were examined in~\cite{Idrissi}.
\end{remark}

\subsection{Acknowledgement} The author is grateful to one of the anonymous referees for pointing out a mistake in an earlier version of this article.

\section{Preliminaries}\label{s:prelim}
\subsection{The Fulton-MacPherson version of the Swiss cheese operad}\label{ss:FM}
 The Fulton-MacPherson operad $\calF_m$ will be used  as a model of the little discs operad $\E_m$. Its 
$j$-th component $\calF_m(j)$ is defined as a compactification of the configuration 
space of~$j$  points in $\R^m$ modulo translations and positive rescalings: $C(j,\R^m)/\R_+\ltimes\R^m$.
One has $\calF_m(0)=\calF_m(1)=*$, while each $\calF_m(j)$, $j\geq 2$, is a smooth (and semi-algebraic) compact
manifold with corners of dimension $mj-m-1$. The strata of $\calF_m(j)$ are in one-to-one correspondence 
with rooted trees whose leaves are bijectively labelled by $1,2,\ldots,j$ and all of whose  internal vertices are 
of arity $\geq 2$. For more details, see~\cite{GJ}.

For a model of the Swiss cheese operad $\SC_m$, we will use its similarly defined Fulton-MacPherson version,
that we denote by $\FS_m$, see~\cite{Voronov, Will}, where this operad was introduced and studied. Its second output components are defined as follows:
\[
\FS_m(i,j)_2=
\begin{cases}
\F_m(j),&\text{if } i=0;\\
\emptyset& \text{if } i\geq 1.
\end{cases}
\]
(In particular, $(\FS_m)_2=\F_m$.) Its each first output component $\FS_m(i,j)_1$ is defined as a compactification of 
the configuration space of~$i$ labelled points in $\R^{m-1}\times\{0\}$ and $j$ points in $\R^{m-1}\times (0,+\infty)$ quotiented out by translations and rescalings:
\[
\left.\Bigl( C(i,\R^{m-1}\times\{0\})\times C(j, \R^{m-1}\times (0,+\infty))\Bigr)\right/\R_+\ltimes \R^{m-1}.
\]
One has $\FS_m(0,0)_1=\FS_m(1,0)_1=\FS_m(0,1)_1=*$ and $(\FS_m)_1=\F_{m-1}$, while for $i+2j\geq 2$, each component $\FS_m(i,j)_1$ is a smooth (and semi-algebraic) compact manifold with corners
of dimension $(m-1)i+mj-m$. For example, $\FS_m(1,1)_1=D^{m-1}$. We agree to label the strata of $(\FS_m)_1=\F_{m-1}$ by trees with solid edges, while those of $(\FS_m)_2=\F_m$ by trees with dashed edges.
The strata of $\FS_m(i,j)_1$, $j\geq 1$, are labelled by trees with two types of edges: solid and dashed, and satisfying the following properties:
\begin{itemize}
\item  output is solid;
\item there are $i$ solid inputs and $j$ dashed inputs
labelled by $1,\ldots,i$ and $1,\ldots,j$, respectively;
\item  arity of each internal vertex is $\geq 1$;
\item  if the arity of an internal  vertex is one, then its incoming edge is dashed and its outgoing edge is solid; 
\item  if a vertex has a solid incoming edge, then its outgoing edge must also be solid. 
\end{itemize}
 As example, $\FS_m(0,2)_1$ is homeomorphic to $S^{m-1}\times [0,+\infty]$ and consists of five strata. The interior is labelled by 
$\tiny\begin{tikzpicture}[baseline=-.6ex, scale=0.6]
\node[int] (w) at (0,0) {};
\node (v) at (0,-0.8) {};
\draw (w) edge (v);
\node (x) at (-0.5,0.7) {};
\node (y) at (0.5,0.7) {};
\draw[dotted,thick] (w) edge (y) edge (x);
\end{tikzpicture}$.  The boundary part $S^{m-1}\times \{+\infty\}$ is labelled by 
$\tiny\begin{tikzpicture}[baseline=-.6ex,scale=0.6]
\node[int] (w) at (0,0.2) {};
\node[int] (v) at (0,-0.25) {};
\node (a) at (0,-0.8) {};
\draw (a) edge (v);
\node (x) at (-0.5,0.7) {};
\node (y) at (0.5,0.7) {};
\draw[dotted,thick] (w) edge (y) edge (x) edge (v);
\end{tikzpicture}$. 
The part $S^{m-1}\times \{0\}$ is a union of two $(m-1)$-discs 
$\tiny\begin{tikzpicture}[baseline=-.6ex, scale=0.6]
\node[int] (w) at (0,-0.25) {};
\node[int] (a) at (-0.19,0.16) {};
\node (v) at (0,-0.8) {};
\draw (w) edge (v);
\node (x) at (-0.5,0.7) {};
\node (y) at (0.5,0.7) {};
\draw[dotted,thick] (w) edge (y) ;
\draw (w) edge (a);
\draw[dotted,thick] (a) edge (x) ;
\end{tikzpicture}$
and 
$\tiny\begin{tikzpicture}[baseline=-.6ex, scale=0.6]
\node[int] (w) at (0,-0.25) {};
\node[int] (a) at (0.19,0.16) {};
\node (v) at (0,-0.8) {};
\draw (w) edge (v);
\node (x) at (0.5,0.7) {};
\node (y) at (-0.5,0.7) {};
\draw[dotted,thick] (w) edge (y) ;
\draw (w) edge (a);
\draw[dotted,thick] (a) edge (x) ;
\end{tikzpicture}$
along the $(m-2)$-sphere 
$\tiny\begin{tikzpicture}[baseline=-.6ex, scale=0.6]
\node[int] (w) at (0,-0.25) {};
\node[int] (a) at (-0.19,0.16) {};
\node[int] (b) at (0.19,0.16) {};
\node (v) at (0,-0.8) {};
\draw (w) edge (v) edge (a) edge (b);
\node (x) at (-0.5,0.7) {};
\node (y) at (0.5,0.7) {};
\draw[dotted,thick] (b) edge (y) ;
\draw[dotted,thick] (a) edge (x) ;
\end{tikzpicture}
$. 

In our arguments it will be convenient to consider the arity one vertices together with the incoming and outgoing edges as a singe edge with the bottom half solid and the 
upper half dashed.  
With this convention, all the trees labeling strata of $\FS_m(0,2)_1$ are of the same shape.

\subsection{Homotopy setting}\label{ss:htpy}
We work in the Reedy model category $\Op$ of reduced topological operads and $\TT_k\Op$ of $k$-truncated reduced operads. We usually deal with two-colored operads. By a
 $k$-truncated operad $\calP$ we mean a collection of components $\calP(i,j)_r$, $i+j\leq k$, $r\in\{1,2\}$, with the composition operations of usual operads whenever such maps are defined.
For the usual one-colored ones, these model category structures were constructed and studied in~\cite[Section~II.8.4]{Fresse}. In particular, it was proved in~\cite{FTW0} that the derived mapping space of (truncated) operads
is the same for the Reedy model structure and the usual projective one~\cite{BM}. We denote by $\Op(-,-)$,$\TT_k\Op(-,-)$, $\Op^h(-,-)$, $\TT_k\Op^h(-,-)$ the corresponding mapping spaces and derived mapping spaces, respectively.

One has a natural truncation functor $\TT_k\colon\Op\to\TT_k\Op$ that forgets the components of combined arity $>k$. To simplify notation instead of $\TT_k\Op(\TT_k\calP,\TT_k\calQ)$ we will be writing $\TT_k\Op(\calP,\calQ)$. 
We will follow the same convention for the derived mapping spaces.

A possibly truncated operad $\calF$ is Reedy cofibrant if and only if its suboperad $\calF^{>0}$ of positive arity components is projectively cofibrant, see~\cite[Theorem~II.8.4.12]{Fresse}. For example, $\calF_m$ and $\FS_m$ are Reedy 
cofibrant.

For any component $\calP(i,j)_r$ of a 2-colored operad, one can define an $(i+j)$-cubical diagram $\calP(i-\bullet,j-\bullet)_r$ obtained by composing $\calP(i,j)_r$ with the arity zero operations. 
Its matching object $\MM\calP(i,j)_r$ is defined as the limit of the subcubical diagram obtained from $\calP(i-\bullet,j-\bullet)_r$ by removing its initial object $\calP(i,j)_r$. 
An operad $\calP$ is Reedy fibrant if for all $i,j\geq 0$, $r\in\{1,2\}$, the map $\calP(i,j)_r\to\MM \calP(i,j)_r$ is a Serre fibration. For any Reedy fibrant operad, $\MM\calP(i,j)_r$ is equivalent to the homotopy limit of the corresponding  
subcubical part of $\calP(i-\bullet,j-\bullet)_r$. We refer to~\cite[Section~II.8.3]{Fresse}, where these facts are stated and proved in the case of one-colored operads. The two-colored generalization is straightforward.

For our proof we also need to consider the Reedy model category $\TT_{k-1,\ell}\Op$ of $(k-1,\ell)$-truncated operads in which components have either combined arity $\leq k-1$ or have arity $k$ with $\leq \ell$ inputs of the second color.
By construction, the truncation functors $\TT_k\colon\Op\to\TT_k\Op$ and $\TT_{k-1,\ell}\colon\Op\to\TT_{k-1,\ell}\Op$ preserve fibrant objects. The former one preserves cofibrant objects, while the latter preserves cofibrant objects
$\calF$ in which $\calF(1,0)_2=\emptyset$ (and hence, by composing with the arity zero operations, all  $\calF(i,j)_2=\emptyset$, $i\geq 1$).

\section{Proof of Theorem~\ref{thm}}\label{s:proof}
In the proof below we assume that $m\geq 2$. The case $m=1$ is proved by an easy modification of this argument. See also Remark~\ref{r:m1}. Note also that~\eqref{eq:th2} follows from~\eqref{eq:th1} by taking $k\to\infty$.

The operad $\FS_m$ is Reedy cofibrant. Without loss of generality assume that $\calP$ is Reedy fibrant. 
We prove \eqref{eq:th1} inducting over $k$. For $k=1$,
\[
\TT_1\Op(\FS_m,\calP)=\Map\bigl(\FS_m(0,1)_1,\calP(0,1)_1\bigr)=\calP(0,1)_1\simeq *.
\]
The last equivalence is by Definition~\ref{d:appropr}(a). On the other hand, $\TT_1\Op(\calF_{m-1},\calP_1)=*$.

Now, assume that the statement holds for $k-1$. Since the left arrow in the following pullback square is an equivalence, so is the right one.
\[
\xymatrix{
\TT_{k-1}\Op(\FS_{m},\calP)\ar[d]_{\simeq}&\TT_{k-1,0}\Op(\FS_m,\calP)\ar[d]^{\simeq}\ar@{->>}_{\rho_0}[l]\\
\TT_{k-1}\Op(\calF_{m-1},\calP_1)& \TT_k\Op(\calF_{m-1},\calP_1).\ar@{->>}[l]
}
\]
Consider the following sequence of fibrations.
\begin{multline}
\TT_{k-1,0}\Op(\FS_m,\calP)\xleftarrow[\rho_1]{\,\simeq\,}
\TT_{k-1,1}\Op(\FS_m,\calP)\xleftarrow[\rho_2]{\,\simeq\,}
\ldots   \\
\xleftarrow[\rho_k]{\,\simeq\,} 
\TT_{k-1,k}\Op(\FS_m,\calP) 
= \TT_k\Op(\FS_m,\calP).
\end{multline}
Each map $\rho_i$, $0\leq i\leq k$, corresponds to an attachment of new free strata in $\FS_m$. We claim that the maps $\rho_i$, $1\leq i\leq k$, 
 are all weak equivalences, which would finish the proof. The fiber of the map $\rho_i$, $1\leq i\leq k-1$, over any
 $b_{i-1}\in\TT_{k-1,i-1}\Op(\FS_m,\calP)$
is the space of
$(\Sigma_{k-i}\times\Sigma_{i})$-equivariant maps $F_i\colon\FS_m(k-i,i)_1\to\calP(k-i,i)_1$ such that the following diagram commutes:
\[
\xymatrix{
\partial\FS_m(k-i,i)_1\ar[d]\ar[r]&\calP(k-i,i)_1\ar@{->>}^\simeq[d]\\
\FS_m(k-i,i)_1\ar[r] \ar@{-->}^{F_i}[ru]  &  \MM\calP(k-i,i)_1,
}
\]
where the horizontal maps are determined by $b_{i-1}$. Since $\calP$ is Reedy fibrant, the right map in the square is a fibration. Moreover, it is an equivalence, which follows from Lemma~\ref{l:cubes} and Definition~\ref{d:appropr}(b). We conclude that any fiber of $\rho_i$ can be identified with a space of sections relative to the boundary of a trivial fibration over $\left.\FS_m(k-i,i)_1\right/(\Sigma_{k-i}\times\Sigma_{i})$.
Thus, the fiber is weakly contractible.


\begin{lemma}\label{l:cubes}
Let $X(\bullet)$ and $Y(\bullet)$ be an $i$-cubical and a $j$-cubical diagrams, respectively, where $i,j\geq 1$. Then the $(i+j)$-cube $X(\bullet)\times Y(\bullet)$ is homotopy 
cartesian.
\end{lemma}
\begin{proof}
One needs to apply \cite[Proposition 1.18]{Goodwillie} twice: first, by considering the fiber cube with respect to some $X(\bullet)$-direction and then, by taking the fiber cube with respect to any $Y(\bullet)$-direction. Equivalently,
one can apply \cite[Lemma~2.2(ii)]{GKK} and the fact that any square of $X(\bullet)\times Y(\bullet)$, whose one side is parallel to $X(\bullet)$ and another one parallel to  $Y(\bullet)$, is a homotopy pullback.
\end{proof}

We finally check that $\rho_k$ is an equivalence. The fiber of $\rho_k$ over any $b_{k-1}\in\TT_{k-1,k-1}\Op(\FS_m,\calP)$ is the space of pairs $(f,F_k)$ of $\Sigma_k$-equivariant maps
\begin{align*}
f\colon\F_m(k)=\FS_m(0,k)_2\to \calP(0,k)_2,\\
F_k\colon \FS_m(0,k)_1\longrightarrow \calP(0,k)_1
\end{align*}
making the following diagrams commute:
\[
\xymatrix{
\partial\FS_m(0,k)_2\ar[d]\ar[r]&\calP(0,k)_2\ar@{->>}[d]\\
\FS_m(0,k)_2\ar[r] \ar@{-->}^{f}[ru]  &  \MM\calP(0,k)_2,
}
\quad
\xymatrix{
\partial\FS_m(0,k)_1\ar[d]\ar[r]&\calP(0,k)_1\ar@{->>}[d]\\
\FS_m(0,k)_1\ar[r] \ar@{-->}^{F_k}[ru]  &  \MM\calP(0,k)_1.
}
\]
The upper horizontal arrow in the first square as well as the lower horizontal arrows in both squares are determined by the choice of $b_{k-1}$. The upper 
horizontal arrow in the second square is determined by~$b_{k-1}$ on all the strata of $\partial\FS_m(0,k)_1$ except one. The remaining stratum $\FS_m(0,1)_1\circ_1\FS_m(0,k)_2$ 
corresponds to  the tree
\[
T_0\,\,\,=\,
\begin{tikzpicture}[baseline=-1ex]
\node at (-0.6,0.8) {$1$};
\node at (-0.2,0.8) {$2$};
\node at (0.2,0.7) {$\ldots$};
\node at (0.6,0.8) {$k$};
\node[int] (w) at (0,0.1) {};
\node[int] (v) at (0,-0.3) {};
\node (a) at (0,-0.8) {};
\draw (a) edge (v);
\node (x) at (-0.6,0.7) {};
\node (y) at (0.6,0.7) {};
\node (z) at (-0.2,0.7) {};
\draw[dotted,thick] (w) edge (y) edge (x) edge (v) edge (z);
\end{tikzpicture}.
\]
Note that it is homeomorphic to $\FS_m(0,k)_2$. On this remaining stratum the upper horizontal map is determined by the image $x\in\calP(0,1)_1$ of $\FS_m(0,1)_1=*$ (which is obviously determined by $b_{k-1}$) and by $f$. 

Let $\partial_0 \FS_m(0,k)_1\subset\partial\FS_m(0,k)_1$ denote the boundary of $\FS_m(0,k)_1$ minus the (open) stratum encoded by the tree $T_0$. We claim that the fiber of~$\rho_k$ 
is equivalent to the space of $\Sigma_k$-equivariant maps $G_k\colon\FS_m(0,k)_1\to \calP(0,k)_1$ making the following diagram commute:
\[
\xymatrix{
\partial_0\FS_m(0,k)_1\ar[d]_\simeq\ar[r]&\calP(0,k)_1\ar@{->>}[d]\\
\FS_m(0,k)_1\ar[r] \ar@{-->}^{G_k}[ru]  &  \MM\calP(0,k)_1.
}
\]
This follows from the fact that  thanks to the weak equivalence~\eqref{eq:gamma_x}, the space of lifts $f\colon\FS_m(0,k)_2\to \calP(0,k)_2$  in the following
diagram is equivalent to the space of lifts $g\colon\FS_m(0,k)_2\to \calP(0,k)_1$:
\[
\xymatrix{
\partial\FS_m(0,k)_2\ar[d]\ar[r]&\calP(0,k)_2\ar@{->>}[d]\ar[rr]_\simeq^{x\circ_1-}&&\calP(0,k)_1\ar@{->>}[d]\\
\FS_m(0,k)_2\ar[r] \ar@{-->}^{f}[ru] \ar@{-->}_(0.6){g}[rrru]  &  \MM\calP(0,k)_2\ar[rr]_\simeq^{x\circ_1-}&& \MM\calP(0,k)_1.
}
\]
The proof is concluded by  the following lemma.

\begin{lemma}\label{l:triv_cof}
The inclusion $i_0\colon\partial_0\FS_m(0,k)_1\to\FS_m(0,k)_1$ is a trivial $\Sigma_k$-cofibration of $\Sigma_k$-cofibrant objects (in the projective model category structure on $\Sigma_k$-spaces).
\end{lemma}

\begin{example}\label{ex:m_1}
$\FS_1(0,k)_1=\Sigma_k\times A_\infty(k+1)$, while $\partial_0\FS_1(0,k)_1=\Sigma_k\times \partial_0 A_\infty(k+1)$, where $ \partial_0 A_\infty(k+1)\subset \partial A_\infty(k+1)$
is the boundary of the associahedron with one open face removed. Hence, $\partial_0 \FS_1(0,k)_1\simeq_{\Sigma_k}\FS_1(0,k)_1\simeq_{\Sigma_k}\Sigma_k$.
\end{example}

\begin{proof}[Proof of Lemma \ref{l:triv_cof}] 
The fact that $i_0$ is a $\Sigma_k$-cofibration with a cofibrant domain holds by \cite[Corollary~5.2]{BM} (which in particular  implies that any cofibrant operad is $\Sigma$-cofibrant).
The fact that~$i_0$ is a weak equivalence follows from Proposition~\ref{p:contr}. 

Denote by $\pi\colon\FS_m(0,k)_1\to \F_m(k)$ the projection induced by the obvious inclusion $C\bigl(k,\R^{m-1}\times (0,+\infty)\bigr)\to C(k,\R^m)$.

\begin{proposition}\label{p:contr}
For every $m\geq 1$, $k\geq 2$, and $z\in \F_m(k)$, both $\pi^{-1}(z)\subset\FS_m(0,k)_1$ and $(\pi\circ i_0)^{-1}\subset \partial_0 \FS_m(0,k)_1$ are contractible.
\end{proposition}

Together with  \cite[Main Theorem]{Smale_vietoris}, this proposition implies that both $\pi$ and $\pi\circ i_0$ are weak equivalences. By the two-out-of-three property, so is~$i_0$.
\end{proof}

\section{Proof of Proposition~\ref{p:contr}}\label{s:prop_proof}
We will give an explicit description of $\pi^{-1}(z)$ and $(\pi\circ i_0)^{-1}(z)$  for any $z\in\F_m(k)$. First, we claim that for any $z$ 
 in the interior of $\F_m(k)$,
\begin{equation}\label{eq:inter}
(\pi\circ i_0)^{-1}(z)\cong\{0\}\subset [0,\infty]\cong\pi^{-1}(z).
\end{equation}
Indeed, consider the map
$
\lambda\colon\FS_m(0,k)_1\to[0,+\infty],
$
defined in the interior by the formula
\[
\lambda\colon [(x_1,\ldots,x_k)] \mapsto  \frac{\underset{1\leq \ell\leq j}{\min} (x_\ell)_m}{\underset{1\leq \ell_1< \ell_2\leq j}\max\bigl( |x_{\ell_2}-x_{\ell_1}|\bigr)}
\]
and extended to the boundary $\partial\FS_m(0,k)_1$ by continuity. It is easy to see that $\lambda|_{\pi^{-1}(z)}$ gives the homeomorphism~\eqref{eq:inter}.
The part $\lambda^{-1}\bigl((0,1)\bigr)\cap\pi^{-1}(z)$ lies in the interior of $\FS_m(0,k)_1$, the point $\lambda^{-1}(+\infty)\cap\pi^{-1}(z)$ lies in the stratum
encoded by~$T_0$, and $\lambda^{-1}(0)\cap\pi^{-1}(z)\in\partial_0\FS_m(0,k)_1$. The latter point lies in the stratum corresponding to the tree

\begin{equation}\label{eq:lambda0}
\begin{tikzpicture}[baseline=-.7ex]
\node at (-1.8,1.5) {$j_1$};
\node at (-1.2,1.5) {$j_2$};
\node at (0,1.5) {$j_r$};
\node at (-0.5,0.9) {$\ldots$};
\node at (0.9,0.9) {$\ldots$};
\node at (0.7,1.5) {$j'_1$};
\node at (1.9,1.5) {$j'_{k-r}$};
\node  (z1) at (-1.8,1.3) {};
\node  (z2) at (-1.2,1.3) {};
\node  (z3) at (0,1.3) {};
\node  (z4) at (0.7,1.3) {};
\node  (z5) at (1.9,1.3) {};
\node[int] (y3) at (0,0.4) {};
\node[int] (y2) at (-0.7,0.4) {};
\node[int] (y1) at (-1.2,0.4) {};
\node[int] (v) at (0,-0.2) {};
\node (root) at (0,-0.9) {};
\draw (v) edge (root) edge (y3) edge (y2) edge (y1);
\draw[dotted,thick] (y1) edge (z1);
\draw[dotted,thick] (y2) edge (z2);
\draw[dotted,thick] (y3) edge (z3);
\draw[dotted,thick] (v) edge (z4) edge (z5);
\end{tikzpicture},
\end{equation}
where
\[
\{j_1,\ldots,j_r\}=\left\{j\in\{1,\ldots,k\}\,|\, (x_j)_m=\min_{1\leq \ell\leq k}(x_\ell)_m\,\right\}=: J_z,
\]
and $[(x_1,\ldots,x_k)]$ is any point in  $\lambda^{-1}\bigl((0,1)\bigr)\cap\pi^{-1}(z)$.  We call $J_z$ the set of {\it bottom elements} of~$z$ since these are the points
in~$z$ with the smallest last coordinate.

On the other hand, the images of the interiors of $\FS_m(0,k)_1$ and of the $T_0$-stratum under~$\pi$ are both  the interior of $\F_m(k)$.
This implies that for any $z\in\partial\F_m(k)$, 
$
(\pi\circ i_0)^{-1}(z)=\pi^{-1}(z).
$
Thus, from now on, we focus on $\pi^{-1}(z)$. 

Note that for $z$ in the interior of $\F_m(i+j)$, $i\geq 1$, its preimage under the map $\FS_m(i,j)_1\to\F_m(i+j)$ induced by the obvious inclusion
\[
C(i,\R^{m-1}\times 0)\times C\left(j,\R^{m-1}\times(0,+\infty)\right)\subset C(i+j,\R^m),
\]
is either a point (if $J_z=\{1,\ldots,i\}$) or the empty set (otherwise).

Consider now $z=(z_v)_{v\in V(T_z)}\in\F_m(k)$ in a stratum corresponding to any tree $T_z$. Here, $V(T)$ denotes the set of internal vertices of a tree~$T$, and
each $z_v$ lies in the interior of $\F_m(|v|)$. It is clear that $\pi^{-1}(z)$ can intersect only the strata labeled by trees $T_z'$ of the same shape as $T_z$,
but with some edges, including the root one,  replaced by solid ones or by half-dashed on top and half-solid on bottom ones. 
By previous considerations, if $T_z'$ is such, the intersection of $\pi^{-1}(z)$ with the closure of the $T_z'$-stratum is either empty or a cube of dimension equal to the number of vertices of $T_z'$ of arity $\geq 2$ and such that the outgoing edge is solid and   all the incoming edges are dashed. 
 As a consequence 
$\pi^{-1}(z)$ 
naturally forms a finite cubical complex.

Define the {\it essential subtree} $S_z$ of $T_z$ by including in it only those edges of $T_z$ that appear as solid in trees $T_z'$  labeling strata with points in~$\pi^{-1}(z)$. We draw $S_z$ with solid edges. It always contains the root edge. Its vertices (we call them {\it essential}) can be described recursively. Except the root vertex, they are all internal. A vertex~$a\in V(T_z)$  is essential if the vertex~$b$  below it is essential and the edge $ab\in J_{z_b}$.
For example, if $z$ lies in the interior of $\F_m(k)$, $S_z$ consists only of two vertices connected by one edge (the root edge). As another example, when $m=1$, $S_z$ is always a linear graph.

Clearly, for every tree $T'_z$  in the preimage of $T_z$, the solid part is connected. Moreover, by construction it is contained in $S_z$ except possibly for leaf edges that 
may happen to be half-solid as, for example, in Figure~\eqref{eq:lambda0}. Based on this observation, it is not hard to see that $\pi^{-1}(z)$ is
homeomorphic to the space $C(S_z)$ of closed connected subsets of  $S_z$  containing the root vertex.

 In the following example, $C(S_z)$ is a union of a segment,
a square and a cube.
$$
S_z=
\begin{tikzpicture}[baseline=4.0ex]
\node[int] (c1) at (-1,1.7) {};
\node[int] (c2) at (0,1.7) {};
\node[int] (b1) at (-.5,1.2) {};
\node[int] (b2) at (.5,1.2) {};
\node[int] (a) at (0,0.7) {};
\node[int] (r) at (0,0) {};
\draw (a) edge (r) edge (b1) edge (b2);
\draw (b1) edge (c1) edge (c2);
\node at (0,-0.3) {root};
\end{tikzpicture}
\hspace{.6in}
\begin{tikzpicture}[baseline=8.0ex, scale=1.4]
\node[int] (b11) at (-1.2,1.5) {};
\node[int] (b12) at (-1.1,2) {};
\node[int] (b13) at (-1.8,2.3) {};
\node[int,draw=gray,fill=gray] (c21) at (-0.7,2) {};
\node[int] (c22) at (-0.6,2.5) {};
\node[int] (c23) at (-1.3,2.8) {};
\node[int] (c2) at (0,1.7) {};
\node[int] (b1) at (-.5,1.2) {};
\node[int] (b2) at (.5,1.2) {};
\node[int] (a) at (0,0.7) {};
\node[int] (r) at (0,0) {};
\draw (a) edge (r) edge (b1) edge (b2);
\draw (b1)  edge (c2) edge (b11) edge (b12);
\draw (c2) edge (b2);
\draw (b13) edge (b11) edge (b12) edge (c23);
\draw (c22) edge (c2) edge (c23) edge (b12);
\draw[dotted,thick] (c21) edge (b11) edge (c23) edge (c2);
\node at (-1.5,3.2) {
        \begin{tikzpicture}[scale=0.3]
         \node[int] (c1) at (-1,1.7) {};
\node[int] (c2) at (0,1.7) {};
\node[int] (b1) at (-.5,1.2) {};
\node[int] (b2) at (.5,1.2) {};
\node[int] (a) at (0,0.7) {};
\node[int] (r) at (0,0) {};
\draw (a) edge (r) edge (b1) edge (b2);
\draw (b1) edge (c1) edge (c2);
        \end{tikzpicture}
    };
    
    \node at (-2.1,2.7) {
        \begin{tikzpicture}[scale=0.3]
         \node[int] (c1) at (-1,1.7) {};
\node[int] (c2) at (0,1.7) {};
\node[int] (b1) at (-.5,1.2) {};
\node[int] (a) at (0,0.7) {};
\node[int] (r) at (0,0) {};
\draw (a) edge (r) edge (b1); 
\draw (b1) edge (c1) edge (c2);
        \end{tikzpicture}
    };
    
    \node at (-0.45,2.85) {
        \begin{tikzpicture}[scale=0.3]
         \node[int] (c1) at (-1,1.7) {};
\node[int] (b1) at (-.5,1.2) {};
\node[int] (b2) at (.5,1.2) {};
\node[int] (a) at (0,0.7) {};
\node[int] (r) at (0,0) {};
\draw (a) edge (r) edge (b1) edge (b2);
\draw (b1) edge (c1) ; 
        \end{tikzpicture}
    };
    
     \node at (-1.35,2.4) {
        \begin{tikzpicture}[scale=0.3]
         \node[int] (c1) at (-1,1.7) {};
\node[int] (b1) at (-.5,1.2) {};
\node[int] (a) at (0,0.7) {};
\node[int] (r) at (0,0) {};
\draw (a) edge (r) edge (b1); 
\draw (b1) edge (c1) ; 
        \end{tikzpicture}
    };
    
      \node at (-1.45,1.6) {
        \begin{tikzpicture}[scale=0.3]
\node[int] (c2) at (0,1.7) {};
\node[int] (b1) at (-.5,1.2) {};
\node[int] (a) at (0,0.7) {};
\node[int] (r) at (0,0) {};
\draw (a) edge (r) edge (b1); 
\draw (b1) edge (c2);
        \end{tikzpicture}
    };
    
       \node at (-.6,1.25) {
        \begin{tikzpicture}[scale=0.3]
\node[int] (b1) at (-.5,1.2) {};
\node[int] (a) at (0,0.7) {};
\node[int] (r) at (0,0) {};
\draw (a) edge (r) edge (b1); 
        \end{tikzpicture}
    };
    
      \node at (0.2,2.1) {
        \begin{tikzpicture}[scale=0.3]
\node[int] (b1) at (-.5,1.2) {};
\node[int] (b2) at (.5,1.2) {};
\node[int] (a) at (0,0.7) {};
\node[int] (r) at (0,0) {};
\draw (a) edge (r) edge (b1) edge (b2);
        \end{tikzpicture}
    };
    
        \node at (.73,1.55) {
        \begin{tikzpicture}[scale=0.3]
\node[int] (b2) at (.5,1.2) {};
\node[int] (a) at (0,0.7) {};
\node[int] (r) at (0,0) {};
\draw (a) edge (r) edge (b2); 
        \end{tikzpicture}
    };
    
           \node at (0.2,.95) {
        \begin{tikzpicture}[scale=0.3]
\node[int] (a) at (0,0.7) {};
\node[int] (r) at (0,0) {};
\draw (a) edge (r);
        \end{tikzpicture}
    };
  
  \node[int] at (0.2,0) {};
  
  \node at (-0.57,1.98) {
        \begin{tikzpicture}[scale=0.3]
\node[int,draw=gray,fill=gray] (c2) at (0,1.7) {};
\node[int,draw=gray,fill=gray] (b1) at (-.5,1.2) {};
\node[int,draw=gray,fill=gray] (b2) at (.5,1.2) {};
\node[int,draw=gray,fill=gray] (a) at (0,0.7) {};
\node[int,draw=gray,fill=gray] (r) at (0,0) {};
\draw[gray] (a) edge (r) edge (b1) edge (b2);
\draw[gray] (b1) edge (c2) ; 
        \end{tikzpicture}
    };

\end{tikzpicture}
\,\,\,= C(S_z)\cong \pi^{-1}(z)
$$
\smallskip

We view $S_z$ as a metric space by assigning length one to each edge.
Let $B_t\in C(S_z)$ denote the set of points of distance $\leq t$ to the root. The following homotopy shows that $C(S_z)$ is contractible:
\[
H\colon C(S_z)\times [0,d]\to C(S_z),\,\,\, (P,t)\mapsto P\cup B_t,
\]
where $d$ is the diameter of $S_z$.

%


\begin{thebibliography}{99}
\bibitem{AT1} G. Arone, V. Turchin. \
On the rational homology of high dimensional analogues of spaces of long knots.\
 \emph{Geom. Topol.}\
 18 (2014), no.~3, 1261-1322.





%

\bibitem{BM}
C. Berger, I. Moerdijk.\
 Axiomatic homotopy theory for operads. \
 \emph{Comment. Math. Helv.} 78, no.~4, 805-831, 2003.


%
%

\bibitem{BW_man_calc}
P. Boavida de Brito, M. Weiss.\
Manifold calculus and homotopy sheaves,\
\emph{Homology Homotopy Appl.} 15 (2013), no.~2, 361-383.



\bibitem{BW_conf_cat} P. Boavida de Brito, M. Weiss.\
 Spaces of smooth embeddings and configuration categories.\
 \emph{J. Topol.} 11 (2018), no.~1, 65-143. 
%
%
%
%
%
%
%
%
%
%
%
%
%
%
%
%
%
%
%
%
%
%

\bibitem{DT} J. Ducoulombier,  V. Turchin.\
 Delooping the functor calculus tower.\ 
\emph{Proc. Lond. Math. Soc.} (3) 124 (2022), no.~6, 772-853.
%
%
%
%
%


\bibitem{Fresse}
B. Fresse.\
 Homotopy of operads and Grothendieck-Teichmüller groups. Part 2. The applications of (rational) homotopy theory methods. \
\emph{Math. Surveys and Monogr.}, 217. American Mathematical Society, Providence, RI, 2017. xxxv+704 pp.

\bibitem{FTW0}
B. Fresse, V. Turchin,  T. Willwacher.\ 
The homotopy theory of operad subcategories.\
\emph{J. Homotopy Relat. Struct.} 13 (2018), no.~4, 689-702.

\bibitem{FTW}
B. Fresse, V. Turchin,  T. Willwacher.\ 
The rational homotopy of mapping spaces of
$E_n$ operads.\
Preprint, 2017,  arXiv:1703.06123v1.

%
%

\bibitem{GJ}
E. Getzler, J. D. S. Jones.\
Operads, homotopy algebra and iterated integrals for double loop spaces.\
Preprint, 1994, arXiv:hep-th/9403055

\bibitem{Goodwillie}
T. G. Goodwillie.\
Calculus. II. Analytic functors.\
\emph{K -Theory} 5 (1991/92), no.~4, 295-332.

\bibitem{GKK}
T. G. Goodwillie, M. Krannich, A. Kupers.\
Stability of concordance embeddings.\
\emph{Proc. Roy. Soc. Edinburgh Sect. A} 154 (2024), no.~6, 1713-1748.
 


\bibitem{GW}
 T. G. Goodwillie, M. Weiss.\ Embeddings from the point of view of immersion theory.~II,\ 
\emph{Geom. Topol.} 3 (1999), 103-118.
%
%
%
%
%
%
%
%
%

\bibitem{HKK}
G. Horel, M. Krannich, A. Kupers.\
Two remarks on spaces of maps between operads of little cubes.\
\emph{High. Struct.} 9 (2025), no.~1, 329-339.


%

\bibitem{Idrissi}
N. Idrissi.\
Formality of a higher-codimensional Swiss-cheese operad.\
\emph{Algebr. Geom. Topol.} 22 (2022), no.~1, 55-111.

\bibitem{IdVas}
N. Idrissi, R. Vasconcellos Vieira.\
 Non-formality of Voronov's Swiss-cheese operads.\
\emph{Q. J. Math.} 75 (2024), no.~1, 63-95.


%
%


\bibitem{KrKu}
M. Krannich, A. Kupers.\
The Disc-structure space. \
\emph{Forum Math. Pi} 12 (2024), Paper No.~e26, 98~pp.

\bibitem{KrKu2}
M. Krannich, A. Kupers.\
$\infty$-operadic foundations for embedding calculus.\
Preprint, 2024,
arXiv:2409.10991.






%
\bibitem{KraRand}
M. Krannich, O. Randal-Williams.\
Diffeomorphisms of discs and the second Weiss derivative of $BTop(-)$,
arXiv:2109.03500.
%
%
%
\bibitem{KupRand}
A. Kupers, O. Randal-Williams.\
On diffeomorphisms of even-dimensional discs,\
\emph{J. Amer. Math. Soc.} 38 (2025), no. 1, 63-178.

%
%
%
%
%
%
%
%
%
%
%

\bibitem{LW}
E. Lindell, T. Willwacher.\
 A(nother) model for the framed little disks operad.\
 \emph{Homology Homotopy Appl.} 25 (2023), no.~1, 265-285.

\bibitem{Liv}
M. Livernet.\
 Non-formality of the Swiss-cheese operad, 
\emph{J. Topol.} 8 no. 4 (2015), 1156-1166.

%
%
%
%
%
%
%
%



%
%
%
%
%

\bibitem{Randal}
O. Randal-Williams.\
Diffeomorphisms of discs.\
\emph{ICM -- International Congress of Mathematicians.}\
 Vol.~4. Sections 5-8, 2856-2878.\
EMS Press, Berlin,  2023.




%
%
%
%
%
%
%
%
%
%
%

\bibitem{Smale_vietoris}
S. Smale.\
 A Vietoris mapping theorem for homotopy.\
 \emph{Proc. Amer. Math. Soc.}~8 (1957), 604-610. 

%
%
%
%

\bibitem{Turchin} 
V. Turchin.\
 Context-free manifold calculus and the Fulton-MacPherson operad,\
\emph{Algebr. Geom. Topol.} 13 (2013), no.~3, 1243-1271.




\bibitem{Voronov}
A. A. Voronov.\
The Swiss-cheese operad.\
{\it  Homotopy invariant algebraic structures} (Baltimore, MD, 1998), 365-73.
\emph{Contemp. Math.}, 239
American Mathematical Society, Providence, RI, 1999.


%
%
%
%
%
%
%
%

\bibitem{Weiss}
 M. Weiss. Embeddings from the point of view of immersion theory.~I,\ 
\emph{Geom. Topol.} 3
(1999), 67-101.

\bibitem{Weiss1}
 M. Weiss.\ 
Truncated operads and simplicial spaces. \
\emph{Tunis. J. Math.}~1 (2019), no.~1, 109-126. 

\bibitem{Will}
T. Willwacher.\
 Models for the $n$-Swiss cheese operads. \
\emph{Trans. London Math. Soc.} 8 (2021), no.~1, 186-205.

%
%
%
%

\end{thebibliography}
\end{document}